\documentclass{amsart}

\usepackage{amssymb,amsmath,amscd,amsthm,enumerate}

\date{\today}

\newcommand{\Z}{{\mathbb Z}}
\newcommand{\R}{{\mathbb R}}
\newcommand{\C}{{\mathbb C}}

\newcommand{\T}{{\mathbb T}}

\newcommand{\SL}{{\mathrm{SL}}}
\newcommand{\SO}{{\mathrm{SO}}}

\newcommand{\Leb}{{\mathrm{Leb}}}

\newcommand{\Hd}{{\mathrm{H}}}
\newcommand{\LP}{{\mathrm{LP}}}

\newtheorem{theorem}{Theorem}[section]
\newtheorem{lemma}[theorem]{Lemma}

\theoremstyle{definition}
\newtheorem*{remark}{Remark}

\theoremstyle{definition}

\theoremstyle{definition}

\theoremstyle{definition}
\newtheorem{claim}[theorem]{Claim}

\theoremstyle{definition}

\allowdisplaybreaks

\numberwithin{equation}{section}

\newcommand{\tr}{\mathrm{tr} }

\newcommand{\eqdef}{\overset{\mathrm{def}}=}

\begin{document}

\title[Limit-Periodic Operators with Zero Measure Spectrum]{Limit-Periodic Continuum Schr\"odinger Operators with Zero Measure Cantor Spectrum}

\author[D.\ Damanik]{David Damanik}

\address{Department of Mathematics, Rice University, Houston, TX~77005, USA}

\email{damanik@rice.edu}

\thanks{D.\ D.\ was supported in part by NSF grants DMS--1067988 and DMS--1361625.}

\author[J.\ Fillman]{Jake Fillman}

\address{Department of Mathematics, Virginia Tech, 225 Stanger Street -- 0123, Blacksburg, VA~24060, USA (Corresponding Author)}

\email{fillman@vt.edu (Corresponding Author)}

\thanks{J.\ F.\ was supported in part by NSF grants DMS--1067988 and DMS--1361625. (Corresponding Author)}

\author[M.\ Lukic]{Milivoje Lukic}

\address{Department of Mathematics, Rice University, Houston TX 77005, USA and Department of Mathematics, University of Toronto, Bahen Centre, 40 St.\ George St., Toronto, Ontario, CANADA M5S 2E4}

\email{mlukic@math.toronto.edu}

\thanks{M.\ L.\ was supported in part by NSF grant DMS--1301582.}

\thanks{The authors would also like to thank the Isaac Newton Institute for Mathematical Sciences, Cambridge, for support and hospitality during the programme ``Periodic and Ergodic Spectral Problems" where part of this work was undertaken.}

\begin{abstract}
We consider Schr\"odinger operators on the real line with limit-periodic potentials and show that, generically, the spectrum is a Cantor set of zero Lebesgue measure and all spectral measures are purely singular continuous. Moreover, we show that for a dense set of limit-periodic potentials, the spectrum of the associated Schr\"odinger operator has Hausdorff dimension zero. In both results one can introduce a coupling constant $\lambda \in (0,\infty)$, and the respective statement then holds simultaneously for all values of the coupling constant.
\end{abstract}

\maketitle



\textbf{MSC2010 Subject Class:} 34L40

\section{Introduction}

We will study spectral characteristics of self-adjoint operators of the form
\[
H_V \phi
=
-\phi'' + V\phi
\]
in $L^2(\R)$, where $V:\R \to \R$ is a bounded, continuous function, known as the \emph{potential}. Our results concern the class of (uniformly) \emph{limit-periodic} potentials, that is, potentials $V$ which are uniform limits of continuous periodic functions on $\R$. Let $\LP = \LP(\R)$ denote the set of uniformly limit-periodic functions $\R \to \R$. Equipped with the $L^\infty$ norm, this is a complete metric space of functions.  It is well known that the spectrum of $H_V$ has a tendency to be a Cantor set whenever $V$ is limit-periodic; compare; for example, \cite{AS81, FL, MC, M, PT88}. Here we show the following result:

\begin{theorem} \label{t:zmdense:coup}
There is a residual subset $\mathcal C \subseteq \LP$ such that $\sigma(H_{\lambda V})$ is a perfect set of zero Lebesgue measure, and $H_{\lambda V}$ has purely singular continuous spectrum for all $V \in \mathcal C$ and all $\lambda > 0$.
\end{theorem}

We will first address the question of zero-measure Cantor spectrum. By the Baire Category Theorem, it suffices to prove the following theorem to demonstrate generic persistence of zero-measure spectrum at arbitrary coupling.

\begin{theorem} \label{t:urdense:coup}
For $R > 0$, $\delta > 0$, and $\Lambda > 1$, let $U_{R,\delta,\Lambda}$ denote the set of $V \in \LP$ with $\Leb(\sigma(H_{\lambda V}) \cap [-R,R]) < \delta$ for all $\Lambda^{-1} \leq \lambda \leq \Lambda$. Then, for all $R,\delta > 0$, and $\Lambda > 1$, $U_{R,\delta,\Lambda}$ is a dense, open subset of $\LP$.
\end{theorem}

Moreover, if we control things more carefully, we can even get spectra of global Hausdorff dimension zero (though this set will only be dense).

\begin{theorem} \label{t:hd0dense}
There is a dense set $\mathcal H \subseteq \LP$ such that $\sigma(H_{\lambda V})$ has Hausdorff dimension zero and such that $H_{\lambda V}$ has purely singular continuous spectrum for all $V \in \mathcal H$ and all $\lambda > 0$.
\end{theorem}

Though the foregoing result is a continuum analog of a known result for discrete Schr\"odinger operators, it is rather striking in this setting since, heuristically speaking, the small coupling and high energy regimes both tend to conspire to ``thicken'' the spectrum, but this construction beats both: one gets spectrum of global Hausdorff dimension zero for small coupling.

Our proofs work by adapting a construction of Avila \cite{avila2009CMP} involving discrete Schr\"odinger operators with limit-periodic potentials to the setting of continuum Schr\"odinger operators.

\medskip

It is interesting to ask whether one can produce quasi-periodic continuum potentials which exhibit zero-measure Cantor spectrum. Of course, there are examples (such as the critical Almost-Mathieu operator) in the discrete setting, but one should keep in mind that the high energy region in the continuum case is analogous to the weak coupling regime in the discrete case. Thus, when looking for evidence for the desired phenomenon in the discrete setting, one really needs to look for quasi-periodic potentials that give rise to zero-measure Cantor spectrum for all non-zero values of the coupling constant. No such examples are presently known. Indeed the only known almost periodic examples of this kind in the discrete case are the ones discussed in \cite{avila2009CMP} and hence we were naturally compelled to work out the continuum analog of that work.

We also note that an even easier question is still open in the discrete case. Is it true that, for fixed irrational frequency $\alpha$, the set of $f \in C(\T,\R)$, for which the discrete Schr\"odinger operator with potential $n \mapsto f(n\alpha)$ has zero-measure Cantor spectrum, is dense, or even residual? This question motivated the work \cite{ADZ14}, where only the following weaker result was shown: for fixed irrational frequency $\alpha$, the set of $f \in C(\T,\R)$, for which the density of states measure is singular, is residual.

\medskip

To give additional motivation for the results above, let us put them in context. It was shown by Fillman and Lukic \cite{FL} that for an explicit dense set of limit-periodic continuum Schr\"odinger operators, the spectrum is homogeneous in the sense of Carleson, and hence in particular of positive Lebesgue measure. Theorem~\ref{t:zmdense:coup} is a companion result, which says that the generic behavior is different. Another perspective on these results is provided by Deift's question about solutions to the KdV equation with almost periodic initial condition; see \cite[Problem~1]{D}. Egorova answered the conjecture in the affirmative for a class of reflectionless limit-periodic potentials with homogeneous spectrum \cite{E} (the same class that is considered in \cite{FL}). Additionally, there has been some recent progress on this question in the case of small quasi-periodic initial data \cite{BDGL, DG}. In particular, the works \cite{BDGL,E} suggest that homogeneity of the spectrum, along with reflectionlessness of the initial condition, may indeed be important to one's ability to show almost periodicity in time for the solution in question, as conjectured by Deift. Indeed, the initial data covered by the works \cite{DG,E} obey these conditions. Thus, in order to explore the limitations of the approach to Deift's question suggested in these recent papers, it is natural to ask if and ``how often'' the necessary conditions are satisfied. The examples provided by Theorem~\ref{t:zmdense:coup}, and especially those provided by Theorem~\ref{t:hd0dense}, are particularly bad from this perspective. Indeed, whenever the spectrum of a continuum Schr\"odinger operator is this small, we are currently very far from a suitable description of the associated isospectral torus, which gives rise to the phase space for the associated KdV evolution, and this prevents us from proving existence, uniqueness, and almost periodicity of the solution of the KdV equation with such initial data. In other words, Theorems~\ref{t:zmdense:coup} and \ref{t:hd0dense} may be viewed as particular challenges to overcome in order to answer Deift's question in full generality.

\medskip

Let us add some comments on the results established here and interesting questions for further study. As pointed out above, our results hold uniformly in the coupling constant. This phenomenon has shown up repeatedly in the limit-periodic theory. Indeed, not a single limit-periodic potential $V$ is known such that the spectral type of the Schr\"odinger operator with potential $\lambda V$ changes as $\lambda$ is varied in the set $\R \setminus \{ 0 \}$. Thus, no phase transitions may be observed. Similarly, the spectral type is always pure and hence there are also no phase transitions as the energy varies in the spectrum. Finally, no change in spectral type may be observed as one varies the element of the hull. In the quasi-periodic theory one can observe changing spectral type in each of these three scenarios. It would therefore be of obvious interest to exhibit limit-periodic examples for which phase transitions occur, or to show that this can never happen.

Another difference between the limit-periodic theory and the quasi-periodic theory we wish to point out is that there is no obvious way to distinguish between regularity classes of sampling functions in the limit-periodic case, whereas this distinction is very important in the quasi-periodic case. Indeed, in the quasi-periodic case there is a very deep understanding of the case of highly regular sampling functions (such as trigonometric polynomials, analytic functions, and Gevrey functions). It is here where one can observe a variety of phase transitions, and in particular the occurrence of absolutely continuous spectrum and pure point spectrum. In fact, recent work by many authors, culminating in Avila's global theory \cite{avilaGlobal1,avilaGlobal2}, has explained these phenomena in great detail in the analytic one-frequency case. On the other hand, the generic behavior for sampling functions that are merely continuous is very similar to the generic behavior in the limit-periodic case; singular continuous spectra dominate in this regime. Of course, the differences between limit-periodic potentials and quasi-periodic potentials have their roots in the topological structure of their respective hulls; specifically, the hull of a quasi-periodic potential will be isomorphic to a finite-dimensional torus, while the hull of a limit-periodic potential will be isomorphic to a solenoid (i.e.\ an inverse limit of circles, which is also sometimes called an odometer). A possible way to discuss suitable analogues of regularity issues in the limit-periodic case could be devised in terms of the speed of approximation by periodic potentials, relative to the periods of the approximants, and hence it would be interesting to extend the work of Pastur and Tkachenko to more general classes \cite{PT84,PT88}; see also \cite{Chul81,MC}.

\medskip

The structure of the paper is as follows. In Section~\ref{sec:prep}, we collect some relevant background which will help in the proofs of Theorems~\ref{t:zmdense:coup}, \ref{t:urdense:coup}, and \ref{t:hd0dense}. In Section~\ref{sec:smallspec}, we describe a construction which enables one to produce periodic Schr\"odinger operators whose spectra are suitably thin for specific ranges of energies and couplings (Lemma~\ref{l:smallspec}). Section~\ref{sec:zm} uses this construction to prove Theorems~\ref{t:zmdense:coup} and \ref{t:urdense:coup}, and Section~\ref{sec:hd0} contains the proof of Theorem~\ref{t:hd0dense}.

\section*{Acknowledgements}

The authors are grateful to the anonymous referee for an excellent report with many helpful comments and observations.

\section{Preparatory Work} \label{sec:prep}

Here, we collect a few technical lemmas which will be used to prove the main theorems. Let us recall the definitions of some of the relevant tools. First, we describe the transfer matrices, which are used to propagate solutions of the time-independent Schr\"odinger equation. Given a potential $V \in C(\R)$, $E \in \C$, and $s,t \in \R$, the associated \emph{transfer matrices} $A_E(s,t) = A_E^V(s,t)$ are uniquely defined by
\[
\begin{pmatrix}
y'(s) \\ y(s)
\end{pmatrix}
=
A_E(s,t)
\begin{pmatrix}
y'(t) \\ y(t)
\end{pmatrix}
\]
whenever $y$ is a solution of the time-independent Schr\"odinger equation
\begin{equation} \label{eq:se}
-y'' + Vy
=
Ey.
\end{equation}
The \emph{Lyapunov exponent}, which tracks the exponential growth of solutions to \eqref{eq:se}, is given by
\[
L(E)
=
L(E,V)
=
\lim_{x \to \infty} \frac{1}{x} \log\|A_E^V(x,0)\|
\]
whenever this limit exists. It is not hard to see that if $V$ is continuous and $T$-periodic, then $L(E,V)$ exists and satisfies
\begin{equation} \label{per:le}
L(E,V)
=
\frac{1}{T} \log \rho(A_E^V(T,0)),
\end{equation}
where $\rho(A)$ denotes the spectral radius of the matrix $A$, i.e., the maximal modulus of an eigenvalue of $A$. Notice that \eqref{per:le} immediately implies that $L$ is a continuous function of $E$ whenever $V$ is periodic. The transfer matrix over a full period which appears on the right hand side of \eqref{per:le} is called the \emph{monodromy matrix} of the corresponding periodic potential. There is more than one possible choice for the monodromy matrix here; clearly, any transfer matrix over a full period will yield the same result in \eqref{per:le}, since all such matrices are conjugate to one another.

\subsection{The IDS for Periodic Operators}

If $V$ is $T$-periodic, denote the associated monodromy matrices by $\Phi_{E}(s) = \Phi_E^V(s) = \Phi_E^V(s;T) := A_E(s+T,s)$, $\Phi_E := \Phi_{E}(0) = A_E(T,0)$, and denote the discriminant by $D(E) := \tr(\Phi_E)$. Recall that $\SL(2,\R)$ acts on the upper half-plane $\C_+ = \{ z \in \C : \mathrm{Im}(z) > 0\}$ via M\"obius transformations, i.e.,
\[
A \cdot z
=
\frac{az + b}{cz + d},
\text{ where }
A
=
\begin{pmatrix} a & b \\ c & d \end{pmatrix} \in \SL(2,\R).
\]
One can easily check that $A \in \SL(2,\R)$ is elliptic (i.e., $\tr(A) \in (-2,2)$) if and only if the M\"obius action of $A$ on $\C_+$ has a unique fixed point. It turns out that there is a remarkable relationship between the integrated density of states, which is given by the average of a certain spectral measure over one period (see \cite[Equation~(10)]{AS83} and its discussion there), and the M\"obius action of the elliptic monodromy matrices. The following formula is \cite[Equation~(17)]{avila2015JAMS}:
\begin{equation} \label{eq:ids:UHP}
\frac{dk}{dE}(E_0)
=
\frac{1}{2\pi T} \int_0^T \frac{dt}{\mathrm{Im}(z_{E_0}(t))}
\end{equation}
for $E_0$ with $D(E_0) \in (-2,2)$, where $k$ denotes the IDS, and $z_{E_0}(t)$ denotes the unique element of $\C_+$ which is fixed by the M\"obius action of $\Phi_{E_0}(t)$. We can use the relation \eqref{eq:ids:UHP} to find a relationship between the (derivative of the) integrated density of states and norms of transfer matrices.

For each $E$ such that $D(E) \in (-2,2)$ and each $t \in \R$, there exists a conjugacy $M_{E}(t) = M_E^V(t) = M_E^V(t;T) \in \mathrm{SL}(2,\R)$ such that
\[
M_{E}(t) \Phi_{E}(t) M_{E}(t)^{-1}
\in
\mathrm{SO}(2,\R).
\]
Of course, $M_{E}(t)$ is not unique, since one may post-compose it with a rotation, but this is the only ambiguity. More specifically, if $\Phi \in \SL(2,\R)$ is elliptic and $A\Phi A^{-1},B \Phi B^{-1} \in \SO(2,\R)$, then one can check that the M\"obius action of $A B^{-1}$ on $\C_+$ fixes $i$, which implies that $A = OB$ for some $O \in \mathrm{SO}(2,\R)$ (since $\mathrm{SO}(2,\R)$ is the stabilizer of $i$ with respect to the action of $\SL(2,\R)$ on $\C_+$).

\begin{lemma} \label{l:ids:tm}
For all $Q, R > 0$, there is a constant $C_0 = C_0(Q,R)$ with the following property. Suppose $V$ is $T$-periodic with $T \geq 1$ and $\|V\|_\infty \leq Q$. Denote the associated discriminant by $D$ and the integrated density of states by $k$. If $D(E_0) \in (-2,2)$ and $|E_0| \leq R$, then
\begin{equation} \label{eq:ids:tm}
\frac{dk}{dE}(E_0)
\geq
\frac{C_0}{T} \int_0^T \! \|M_{E_0}(t)\|^2 \, dt.
\end{equation}
\end{lemma}

In the course of the proof, we will use the following solution estimates from \cite[Lemma~3.1]{simon96PAMS}
\begin{lemma}\label{l:simon96}
For all $Q,R > 0$, there is a constant $C_1 = C_1(Q,R) > 0$ such that if $u$ satisfies $-u'' + Vu = Eu$ with $|E| \leq R$ and $\| V \|_\infty \leq Q$, then
\[
|u'(x)|^2
\leq
C_1 \int_{x-1}^{x+1} \! |u(t)|^2 \, dt
\]
for all $x \in \R$.
\end{lemma}

\begin{proof}[Proof of Lemma~\ref{l:ids:tm}]

Since $M_E(t)$ is unique modulo left-multiplication by an element of $\mathrm{SO}(2,\R)$, its Hilbert--Schmidt norm is independent of the choice of conjugacy. Since
\[
M_{E}(t)
=
\mathrm{Im}(z_E(t))^{-1/2}
\begin{pmatrix}
1 & - \mathrm{Re}(z_E(t)) \\
0 & \mathrm{Im}(z_E(t))
\end{pmatrix}
\]
clearly furnishes an example of a matrix which conjugates $\Phi_E(t)$ to a rotation, we may explicitly compute the Hilbert--Schmidt norm of $M_{E}(t)$ via
\begin{equation} \label{eq:hsnorm}
\| M_{E}(t) \|_2^2
=
\frac{1+|z_E(t)|^2}{\mathrm{Im}(z_E(t))}.
\end{equation}
With $\theta$ chosen such that $2\cos\theta = D(E)$, there are solutions $\psi_\pm$ of $H \psi = E\psi$ such that $\psi_\pm(x+T) = e^{\pm i \theta}\psi_\pm(x)$ for all $x \in \R$ (indeed, we may take $\psi_- = \overline\psi_+$). Then, the fixed points of the M\"obius action of $\Phi_E(t)$ are precisely $\psi_\pm'(t)/\psi_\pm(t)$.  We choose $\psi \in \{\psi_+,\psi_-\}$ so that $\mathrm{Im}(\psi'(t)/\psi(t)) > 0$, and hence $z_E(t) = \psi'(t)/\psi(t)$.

Applying Lemma~\ref{l:simon96}, Fubini's theorem, and the hypothesis $T\geq 1$, we observe:
\begin{align*}
\int_0^T \! |\psi'(t)|^2 \, dt
& \leq
C_1 \int_0^T \! \int_{t-1}^{t+1} \! |\psi(s)|^2 \, ds \, dt \\
& \leq
C_1 \int_{-T}^{2T} \! \int_{s-1}^{s+1} \! |\psi(s)|^2 \, dt \, ds \\
& =
2C_1 \int_{-T}^{2T} |\psi(s)|^2 \, ds.
\end{align*}
Consequently, if we denote the Wronskian of $\psi$ and $\overline \psi$ by $W = W(\psi,\overline\psi) := \psi'\overline\psi - \psi\overline{\psi'}$, we obtain
\begin{align*}
\int_{0}^{T} \! \frac{dt}{\mathrm{Im}(z_E(t))}
& =
\frac{1}{3} \int_{-T}^{2T} \! \frac{dt}{\mathrm{Im}(z_E(t))} \\
& =
\frac{1}{3} \int_{-T}^{2T} \! \frac{2i}{W} |\psi(t)|^2 \, dt \\
& \geq
\frac{1}{6C_1+3} \int_{0}^{T} \! \frac{2i}{W} \left( |\psi(t)|^2 + |\psi'(t)|^2 \right)\, dt \\
& =
C_0 \int_0^T \! \frac{1+|z_E(t)|^2}{\mathrm{Im}(z_E(t))} \, dt,
\end{align*}
where $C_0 := \frac{1}{6C_1+3}$. Using \eqref{eq:hsnorm}, this yields the conclusion of the lemma with the Hilbert-Schmidt norm on the right hand side of \eqref{eq:ids:tm}. For all matrices $A$, one has $\| A \| \leq \| A \|_2$ by the Cauchy-Schwarz inequality, so the lemma is proved.
\end{proof}

\subsection{Band-counting in Finite Energy Windows}

We will also need the following elementary estimate on the number of bands that one may observe in a finite energy window.

\begin{lemma} \label{l:bandcount}
If $V \in C(\R)$ is $T$-periodic, then $[-R,R]$ intersects at most
\[
\frac{T}{\pi} \sqrt{R+\|V\|_\infty} + 1
\]
bands of $\sigma(H_V)$ for each $R > 0$.
\end{lemma}

\begin{proof}
Regard the free operator $H_0 = -\Delta$ as a $T$-periodic operator. Listed in ascending order, the periodic and antiperiodic eigenvalues of $H_0$ on $L^2([0,T])$ are
\[
E_{n}
=
\frac{n^2 \pi^2}{T^2},
\quad
n \geq 0.
\]
Let $Q = \|V\|_\infty$, and choose $n \in \Z_0$ maximal with $E_{n} \leq R + Q$. By standard eigenvalue perturbation theory, at most $n+1$ bands of $\sigma(H_V)$ touch $[-R,R]$. Since $E_n \leq R + Q$, we have
\[
n
\leq
\frac{T}{\pi} \sqrt{R+Q},
\]
which proves the lemma.
\end{proof}

\section{The Measure of the Spectrum in Finite Energy Windows} \label{sec:smallspec}

We may combine the ingredients of Section~\ref{sec:prep} to construct periodic operators whose spectra are exponentially small (relative to the period) in finite energy regions for compact ranges of coupling constants which are bounded away from zero.

\begin{lemma} \label{l:smallspec}
Suppose $V$ is a $T$-periodic potential, $\varepsilon > 0$, and $\Lambda, R > 1$. There exists $N_0 \in \Z_+$ such that for all $N \in \Z_+$ with $N \geq N_0$, there exists a $\widetilde T := NT$-periodic potential $\widetilde V = \widetilde V_N$ such that $\|V - \widetilde V\|_\infty < \varepsilon$, and
\[
\Leb(\sigma(H_{\lambda \widetilde V}) \cap [-R,R])
\leq
e^{-\widetilde T^{1/2}}
\]
for all $\lambda \in [\Lambda^{-1},\Lambda]$.
\end{lemma}

\begin{proof}
The construction works by first perturbing $V$ to produce a family of potentials which are very close to $V$, and whose resolvent sets cover $[-R,R]$. Thus, for every $E \in [-R,R]$, one of these new potentials will have $L(E) > 0$. We then form a new potential by concatenating these finite families over long blocks and suitably connecting them. Positive exponents over sub-blocks enable us to produce growth of transfer matrices, and we then parlay growth of transfer matrices into upper bounds on band lengths via Lemma~\ref{l:ids:tm}. The details follow.

Denote $I = [\Lambda^{-1}, \Lambda]$. First, choose $N' > 1/T$ large enough that the maximal distance between $N'$-break points of $\lambda V$ contained in $[-R,R]$ is less than $\varepsilon/9$ for all $\lambda \in I$, where an $N'$-break point of $\lambda V$ is a (possibly degenerate) band edge of $\sigma(H_{\lambda V})$, viewed as a $T' := N' T$-periodic operator.

By \cite{simon76} and compactness, there are finitely many potentials $V_1',V_2',\ldots,V_m' \in B_{\varepsilon/(9\Lambda)}(V)$ which are $T'$-periodic and such that for every $\lambda \in I$, there is a $1 \leq j \leq m$ such that $\lambda V_j'$ has all gaps open. More specifically, for each $\lambda_0 > 0$, there is a potential $q$ within $\varepsilon/(9\Lambda)$ of $V$ which is $T'$-periodic and such that $\lambda_0 q$ has all gaps open; since gaps will remain open for $\lambda q$ with $\lambda$ in a suitably small neighborhood of $\lambda_0$, we may pass to finitely many perturbations by using compactness of $I$. Given $1 \leq j \leq m$ and $\lambda \in I$, we have $\| \lambda V_j' - \lambda V\|_\infty < \varepsilon/9$ and the distance between $N'$-break points of $\lambda V$ is less than $\varepsilon / 9$; thus,
\[
\Leb(J\cap [-R,R])
<
\varepsilon/3
\]
for all bands $J$ of $\sigma(H_{\lambda V_j'})$.
\newline

\begin{claim} \label{cl:posle}
There is a finite set $\mathcal F = \{W_1,\ldots,W_\ell\} \subseteq B_{\varepsilon}(V)$ of $T'$-periodic potentials such that
\begin{equation} \label{eq:resolventcover}
[-R,R] \cap \bigcap_{j=1}^\ell
\sigma(H_{\lambda W_j}) = \emptyset
\end{equation}
for all $\lambda \in I$.
\end{claim}

\begin{proof}[Proof of Claim] Given $\lambda_0 \in I$, choose $j$ so that $\sigma(H_{\lambda_0 V_j'})$ has all spectral gaps open, and let $\gamma_0$ denote the minimal length of a gap of $\sigma(H_{\lambda_0 V_j'})$ which intersects $[-R,R]$. Now, put
\[
\gamma
=
\min\left( \frac{\varepsilon}{3}, \frac{\gamma_0}{2\Lambda} \right),
\quad
k
=
\left\lceil \frac{\varepsilon}{3\gamma} \right\rceil,
\]
and define new potentials $U_{-k}, \ldots U_{k}$ by $U_i = V_j' + i\gamma$. Then, it is easy to see that each $U_i$ is in $B_\varepsilon(V)$ and that the resolvent sets of $H_{\lambda_0 U_{-k}}, \ldots, H_{\lambda_0, U_k}$ cover $[-R,R]$. Thus, \eqref{eq:resolventcover} holds for this family and $\lambda = \lambda_0$. Since this will also hold for the same finite family and for $\lambda$ within a neighborhood of $\lambda_0$, the claim follows by compactness of $I$.
\end{proof}

By Claim~\ref{cl:posle} and \eqref{per:le}, we have
\begin{equation} \label{eq:pos:avgLE}
\min_{|E| \leq R} \min_{\lambda \in I} \max_{1 \leq j \leq \ell} L(E,\lambda W_j)
\eqdef
\eta
>
0,
\end{equation}
by continuity of the Lyapunov exponent in the periodic setting. Now, suppose that $N$ is sufficiently large. To construct the desired potential, choose $\widetilde N \in \Z_+$ maximal with $\ell N' (\widetilde N + 2)  \leq N$,  define $\widetilde T = N T$, and generate a new $\widetilde T$-periodic potential $\widetilde V = \widetilde V_N$ by concatenating each $W_j$ a total of $\widetilde N+1$ times, and forming continuous connections which are uniformly close to $V$. More specifically, denote $s_j = j(\widetilde N + 2) T'$ for each integer $0 \leq j \leq \ell$, and define $\widetilde V$ on $[0,NT]$ by
\[
\widetilde V(x)
=
\begin{cases}
W_j(x) & x \in [s_{j-1}, s_j - T'] \\
\varphi_j(x) & x \in [s_j - T', s_j], \, 1 \leq j \leq \ell - 1 \\
\varphi_\ell(x) & x \in [s_\ell - T', NT]
\end{cases}
\]
When $1 \leq j \leq \ell - 1$, $\varphi_j$ is chosen to be a continuous  function on $[s_j - T',s_j]$ with
\[
\varphi_j(s_j-T')
=
W_j(T'),
\quad
\varphi_j(s_j)
=
W_{j+1}(0),
\quad
\sup_{x \in [s_j - T',s_j]} |\varphi_j(x) - V(x)| < \varepsilon.
\]
Similarly, $\varphi_\ell$ is continuous with $\varphi_\ell(s_\ell-T') = W_\ell(T')$, $\varphi_\ell(NT) = W_1(0)$, and $\| \varphi_\ell - V\|_\infty < \varepsilon$.

Now, suppose $E \in [-R,R]$ and $\lambda \in I$ are such that $\widetilde D_\lambda(E) \in (-2,2)$, where $\widetilde D_\lambda$ denotes the discriminant of $H_{\lambda \widetilde V}$. By \eqref{eq:pos:avgLE}, there is $j \in \{1,\ldots,\ell\}$ such that $L(E,\lambda W_j) \geq \eta$. But then the associated transfer matrices over subintervals of $[s_{j-1}, s_j - T']$ of length $\widetilde N T$ are exponentially large. More specifically, we have
\begin{equation} \label{eq:simplele}
\begin{split}
\|  A_E^{\lambda \widetilde V}(s_j+t-2T' , s_{j-1}+t) \|
& \ge
\rho( A_E^{\lambda W_j}(t+T',t)^{\widetilde N} ) \\
& =
\rho( A_E^{\lambda W_j}(t+T',t))^{\widetilde N} \\
& =
e^{\widetilde N  T' L(E,\lambda W_j) } \\
& \geq
e^{\widetilde T \eta/(2\ell)}.
\end{split}
\end{equation}
for all $t \in [0,T']$; we have used \eqref{per:le}. Notice that the last step requires $N$ sufficiently large to get $\widetilde N \geq \frac{1}{2} (\widetilde N + 3)$. We can see that the estimate above implies lower bounds on the norms of the matrices which conjugate the monodromy matrices into rotations. More specifically, with $\Phi = \Phi_E^{\lambda \widetilde V}(s_{j-1} + t;\widetilde T) $, we have $X \Phi X^{-1} \in \SO(2,\R)$ for
\begin{align*}
X
& =
M_E^{\lambda \widetilde V}(s_{j-1}+t;\widetilde T), \\
X
& =
M_E^{\lambda \widetilde V}(s_j+t-2T'; \widetilde T) A_E^{\lambda \widetilde V}(s_j + t - 2T',s_{j-1} +t),
\end{align*}
by periodicity of $\widetilde V$ and definition of $M_E$; more specifically, $\widetilde T$-periodicity of $\widetilde V$ allows one to conclude that $A_E^{\lambda \widetilde V}(s_j + t - 2T', s_{j-1} + t)$ conjugates $\Phi$ to $\Phi_E^{\lambda \widetilde V}(s_j + t -2T';\widetilde T)$, since
\[
A_E^{\lambda \widetilde V}(s_j + t - 2T', s_{j-1} + t)
=
A_E^{\lambda \widetilde V}(s_j + t - 2T'+\widetilde T, s_{j-1} + t+\widetilde T),
\] 
and $\Phi_E^{\lambda \widetilde V}(s_j + t -2T';\widetilde T)$ is then conjugated to a rotation by $M_E^{\lambda \widetilde V}(s_j + t - 2T'; \widetilde T)$. Since conjugacies of elliptic matrices to rotations are unique modulo left-multiplication by elements of $\mathrm{SO}(2,\R)$, we have
\begin{equation} \label{eq:conjrel1}
M_E^{\lambda \widetilde V}(s_j+t-2T'; \widetilde T) A_E^{\lambda \widetilde V}(s_j + t - 2T',s_{j-1} +t)
=
O M_E^{\lambda \widetilde V}(s_{j-1}+t; \widetilde T)
\end{equation}
for some rotation $O = O(E,\lambda, t) \in \SO(2,\R)$. Using the lower bound on the norm of $A_E^{\lambda \widetilde V}(s_j+t-2T',s_{j-1}+t)$, \eqref{eq:conjrel1} implies
\begin{equation} \label{eq:conjbounds}
\max(\| M_{E}^{\lambda \widetilde V}(s_j + t -2 T'; \widetilde T)\|, \| M_{E}^{\lambda \widetilde V}(s_{j-1}+t; \widetilde T) \|)
\geq
e^{\widetilde T\eta/(4\ell)}
\end{equation}
for all $t \in [0,T']$. Notice that this uses $\|M^{-1}\| = \|M\|$ for $M \in \SL(2,\R)$ (which follows from the singular value decomposition). A bit more precisely, if one has $A = M_1^{-1} O M_2$ with $M_j \in \SL(2,\R)$ and $O \in \SO(2)$, then one cannot simultaneously have $\|M_1\|, \|M_2\| < \|A\|^{1/2}$. 

These estimates are uniform in $\lambda \in I$ and $E \in [-R,R] \cap \sigma(H_{\lambda \widetilde V})$. Consequently, for any band $J \subseteq \sigma(H_{\lambda \widetilde V})$, we have
\begin{equation}\label{e.bandest}
\Leb(J \cap [-R,R])
\leq
C e^{-\widetilde T \eta /(4\ell)}
\end{equation}
by Lemma~\ref{l:ids:tm}, where $C$ denotes a constant which depends only on $R$ and $Q := \Lambda(\|V\|_\infty + \varepsilon)$. We have also used that $dk(J) = 1/\widetilde T$ for any band $J$ of $\sigma(H_{\lambda \widetilde V})$, where $dk$ denotes the associated IDS. 
Since all potentials in question are uniformly bounded (by $\Lambda(\|V_0\|_\infty + \varepsilon)$) and $R \geq 1$, Lemma~\ref{l:bandcount} implies that the number of bands of $\sigma(H_{\lambda \widetilde V})$ which touch $[-R,R]$ is bounded above by
\[
\frac{1}{\pi}\widetilde T\sqrt{R + \Lambda(\|V\|_\infty+\varepsilon)}
=
\frac{1}{\pi} \widetilde T\sqrt{R+Q}.
\]
Thus, by \eqref{e.bandest},
\begin{equation} \label{eq:spec:expsmall}
\Leb(\sigma(H_{\lambda \widetilde V}) \cap [-R,R])
\leq
C\widetilde T e^{-\eta \widetilde T /(4\ell)}.
\end{equation}
Since $\widetilde T = NT$ and $C$ only depends on $R$ and $Q$, we may choose $\widetilde N$ sufficiently large (hence $N$ sufficiently large) and make the right hand side of \eqref{eq:spec:expsmall} smaller than $e^{-\widetilde T^{1/2}}$. Since $\|V - \widetilde V\|_\infty < \varepsilon$, the lemma is proved.
\end{proof}

\begin{remark}
Let us comment briefly on the relationship between the  proof of Lemma~\ref{l:smallspec} and the arguments in \cite{avila2009CMP}. The primary difference between our arguments and those of \cite{avila2009CMP} is that we do not attempt to push positive Lyapunov exponents through to the limit. We use growth of transfer matrices purely as a means to control the size of the spectrum (via Lemma~\ref{l:ids:tm}). Consequently, for our purposes, it suffices to consider ``local'' growth behavior of the transfer matrices of $\widetilde V$; more specifically, we only need to produce growth within subblocks of commuting matrices, where one has a simple relationship between the spectral radius and the norm (cf.\ \eqref{eq:simplele}).

If one wishes to obtain a global understanding of transfer matrix growth and control the Lyapunov exponent of $\widetilde V$, then one must deal with concatenated blocks of noncommuting matrices, and the simple Lyapunov behavior exploited in \eqref{eq:simplele} may break down. In this situation, one must produce analogs of \cite[Claim~3.6]{avila2009CMP} and \cite[Claim~3.7]{avila2009CMP} to produce the necessary global growth of transfer matrices.
\end{remark}

\section{Singular Continuous Spectrum of Zero Lebesgue Measure} \label{sec:zm}

\begin{proof}[Proof of Theorem~\ref{t:urdense:coup}]
Let $R$, $\delta$, and $\Lambda$ be given. We first show that $U_{R,\delta,\Lambda}$ is dense in $\LP$. To that end, let $V\in \LP$ be $T$-periodic, and let $\varepsilon > 0$. Choose $N$ large enough that  $e^{-\sqrt{NT}} < \delta$ and Lemma~\ref{l:smallspec} applies, and then let $\widetilde V_N$ be the potential given by the conclusion of the lemma with the same choices of $\varepsilon$, $\Lambda$, and $R$. Evidently, $\widetilde V_N \in U_{R,\delta,\Lambda} \cap B_\varepsilon(V)$, so we are done (since periodic potentials are dense in $\LP$).
\newline

It remains to be seen that $U_{R,\delta,\Lambda}$ is open in $\LP$. Suppose $V \in U_{R,\delta,\Lambda}$. By compactness of $I := [\Lambda^{-1},\Lambda]$, it suffices to show that, for every $\lambda \in I$, there exist $\tau = \tau(\lambda) > 0$ and $r = r(\lambda) > 0$ such that
\[
\Leb(\sigma(H_{\lambda' V'}) \cap [-R,R])
<
\delta
\]
whenever $\lambda' \in I$ and $V' \in \LP$ satisfy $|\lambda - \lambda'| < \tau$ and $\| V - V' \|_\infty < r$. To see why such $\tau$ and $r$ exist, fix $\lambda \in I$, and choose a cover of $\sigma(H_{\lambda V}) \cap [-R,R]$ by open intervals $I_1,\ldots, I_n$ such that $\sum_{j=1}^n |I_j| < \delta$ (which we may do by compactness of $\sigma(H_{\lambda V}) \cap [-R,R]$). Choose $\varepsilon > 0$ small enough that
\[
B_\varepsilon(\sigma(H_{\lambda V}) \cap [-R,R])
\subseteq
\bigcup_{j=1}^n I_j,
\quad
\text{and}
\quad
\sum_{j=1}^n|I_j| + 2\varepsilon < \delta,
\]
where $B_\varepsilon(S)$ denotes the $\varepsilon$-neighborhood of the set $S$. Now, take
\[
\tau
=
\tau(\lambda)
=
\frac{\varepsilon}{2\|V\|_\infty},
\quad
r
=
r(\lambda)
=
\frac{\varepsilon}{2\Lambda},
\]
and suppose $|\lambda - \lambda'| < \tau$ and $\|V - V'\|_\infty < r$; since $\|\lambda V - \lambda ' V' \|_\infty < \varepsilon$, we have
\[
\sigma(H_{\lambda' V'})
\subseteq
B_\varepsilon(\sigma(H_{\lambda V}))
\]
by the usual $L^\infty$ eigenvalue perturbation theory. Consequently,
\[
\Leb(\sigma(H_{\lambda' V'}) \cap [-R,R])
\leq
\sum_{j=1}^n |I_j|
+
2\varepsilon
<
\delta.
\]
Note that the second term originates because new spectrum might ``creep in'' at the edges of the interval $[-R,R]$. Thus, we have proved that $\tau(\lambda)$ and $r(\lambda)$ satisfy the desired properties.
\end{proof}

We can use the foregoing theorem and Baire category to produce generic singular continuous spectrum supported on a set of zero Lebesuge measure. One still needs to exclude eigenvalues on a generic set, but this is easy to do using Gordon methods; we provide the details for the reader's convenience.

\begin{proof}[Proof of Theorem~\ref{t:zmdense:coup}]
By Theorem~\ref{t:urdense:coup}, $U_{R,\delta,\Lambda}$ is a dense open subset of $\LP$ for all $R,\Lambda,\delta$. Now, take a trio of sequences $R_n, \Lambda_n \to \infty$ and $\delta_n \to 0$; by the Baire Category Theorem,
\[
\mathcal Z
=
\bigcap_{n = 1}^\infty U_{R_n,\delta_n, \Lambda_n}
\]
is a dense $G_\delta$ in $\LP$ such that $\Leb(\sigma(H_{\lambda V})) = 0$ for all $V \in \mathcal Z$ and all $\lambda > 0$.

Next, let $\mathcal G \subseteq \LP$ denote the set of \emph{Gordon potentials} in $\LP$. More specifically, $V \in \mathcal G$ if and only if $V \in \LP$ and there exist $T_k \to \infty$ such that
\[
\lim_{k \to \infty} C^{T_k} \max_{|x| \leq T_k} |V(x) - V(x+T_k)|
=
0
\]
for all $C > 0$. It is easy to see that $\mathcal G$ is dense in $\LP$ (as it contains all periodic potentials). Moreover, one can check that $\mathcal G$ is a $G_\delta$. A bit more concretely, for $n, m \in \Z_+$, denote
\[
\mathcal O_{n,m}
=
\left\{
V \in \LP : \exists T \in (n-1,n+1) \text{ with } \max_{|x| \leq T} |V(x) - V(x+T)| < m^{-T}
\right\}.
\]
It is straightforward to check that $\mathcal O_{n,m}$ is open in $\LP$ and that
\[
\mathcal G
=
\bigcap_{N\geq 1} \bigcup_{n \geq N} \bigcup_{m \geq N} \mathcal O_{n,m}.
\]

Since $\sigma_{\mathrm{pp}}(H_{\lambda V}) = \emptyset$ for every $V \in \mathcal G$ and every $\lambda > 0$ \cite{G76}, we obtain the desired result with $\mathcal C = \mathcal G \cap \mathcal Z$.
\end{proof}

\section{Zero Hausdorff Dimension} \label{sec:hd0}

Here, we will prove Theorem~\ref{t:hd0dense}. For the convenience of the reader, and to establish notation, let us briefly recall how Hausdorff measures and dimension are defined; for further details, \cite{Falconer} supplies an inspired reference.

Given a set $S \subseteq \R$ and $\delta > 0$, a $\delta$-\emph{cover} of $S$ is a countable collection of intervals $\{I_j\}$ such that $S \subseteq \bigcup_j I_j$ and $\Leb(I_j) < \delta$ for each $j$. Then, for each $\alpha \geq 0$, one defines the $\alpha$-dimensional \emph{Hausdorff measure} of $S$ by
\[
h^\alpha(S)
=
\lim_{\delta \downarrow 0} \inf_{\delta\text{-covers}} \sum_j \Leb(I_j)^\alpha.
\]
For each $S \subseteq \R$, there is a unique $\alpha_0 \in [0,1]$ such that
\[
h^\alpha(S)
=
\begin{cases}
\infty & 0 \leq \alpha < \alpha_0 \\
0      & \alpha_0 < \alpha
\end{cases}
\]
We denote $\alpha_0 = \dim_{\Hd}(S)$ and refer to this as the \emph{Hausdorff dimension} of the set $S$.

\begin{proof}[Proof of Theorem~\ref{t:hd0dense}]

Let $V_0 \in \LP$ be $T_0$-periodic, and suppose $\varepsilon_0 > 0$.  We will construct a sequence $(V_n)_{n=1}^\infty$ consisting of periodic potentials so that $V_\infty = \lim_n V_n$ satisfies $\|V_0 - V_\infty\|_\infty  < \varepsilon_0$ and
\[
h^\alpha(\sigma(H_{\lambda V_\infty}))
=
0
\]
for all $\lambda > 0$ and all $\alpha>0$; evidently, this suffices to show that $\sigma(H_{\lambda V_\infty})$ has Hausdorff dimension zero for all $\lambda > 0$. Throughout the proof, $H_{n,\lambda} = -\Delta + \lambda V_n$ and $\Sigma_{n,\lambda} = \sigma(H_{n,\lambda})$ for $1 \leq n \leq \infty$, $\lambda > 0$.
\newline

Denote $\Lambda_n = r_n = 2^n$ (in general, one may take any pair of sequences tending to $\infty$ not too quickly). Take $\varepsilon_1 = \varepsilon_0/2$. By Lemma~\ref{l:smallspec}, we may construct a $T_1$-periodic $V_1$ with $T_1 > 1$, $\|V_0 - V_1 \|_\infty < \varepsilon_1$, and for which
\[
\delta_1
:=
\max_{\Lambda_1^{-1} \leq \lambda \leq \Lambda_1} \Leb(\Sigma_{1,\lambda} \cap [-r_1,r_1])
<
e^{-T_1^{1/2}}.
\]
Having constructed $V_{n-1}$, $\delta_{n-1}$, and $\varepsilon_{n-1}$, let
\begin{equation}\label{eq:nextepsdef}
\varepsilon_{n}
=
\min\left(\frac{\varepsilon_{n-1}}{2},
\frac{1}{2} n^{-T_{n-1}},
\frac{\delta_{n-1}}{4\Lambda_{n-1}} \right).
\end{equation}
By Lemma~\ref{l:smallspec}, we may construct a $T_{n} := N_n T_{n-1}$-periodic potential $V_{n}$ with $\|V_n - V_{n-1}\| < \varepsilon_{n}$  such that
\begin{equation} \label{eq:smallspec}
\delta_{n}
:=
\max_{\Lambda_n^{-1} \leq \lambda \leq \Lambda_n} \Leb(\Sigma_{n,\lambda} \cap [-r_n,r_n])
<
e^{-T_{n}^{1/2}}.
\end{equation}
Clearly, $V_\infty = \lim_{n\to\infty}V_n$ exists and is limit-periodic. From the first condition in \eqref{eq:nextepsdef}, we deduce
\[
\| V_0 - V_\infty\|
<
\sum_{j = 1}^\infty \varepsilon_j
\leq
\sum_{j=1}^\infty 2^{-j} \varepsilon_0
=
 \varepsilon_0,
\]
so $V_\infty \in B_{\varepsilon_0}(V_0)$. Similarly, using the first and second conditions in \eqref{eq:nextepsdef}, we observe that
\[
\|V_n - V_\infty\|_\infty
<
n^{-T_n}
\]
for each $n \geq 2$. From this, it is easy to see that $\lambda V_\infty$ is a Gordon potential for every $\lambda > 0$. In particular, $H_{\lambda V_\infty}$ has purely continuous spectrum for all $\lambda > 0$. Thus, it remains to show that the spectrum has Hausdorff dimension zero. The key observation in this direction is that \eqref{eq:nextepsdef} yields
\begin{equation} \label{eq:taildist}
\|\lambda V_n - \lambda V_\infty\|
\leq
\Lambda_n \sum_{j = n + 1}^\infty \varepsilon_j
<
\sum_{k = 2}^\infty 2^{-k} \delta_n
=
\delta_n/2
\end{equation}
for all $n \in \Z_+$ and all $\Lambda_n^{-1} \leq \lambda \leq \Lambda_n$.

\begin{claim} \label{cl:hdim}
For all $j \in \Z_+$, and all $\lambda > 0$, $\dim_{\mathrm H}(\Sigma_{\infty,\lambda} \cap [-r_j,r_j]) = 0$.
\end{claim}

\begin{proof}[Proof of Claim]
Let $j \in \Z_+$, $\delta > 0$, $\lambda > 0$, and $\alpha > 0$ be given. Choose $n \geq j$ for which $2\delta_n < \delta$ and $\lambda \in [\Lambda_n^{-1},\Lambda_n]$. Then, by \eqref{eq:taildist}, the $\delta_n/2$-neighborhood of $\Sigma_{n,\lambda} \cap [-r_n,r_n]$ together with the intervals $[-r_n,-r_n + 2\delta_n]$ and $[r_n - 2\delta_n,r_n]$ comprises a $\delta$-cover of $\Sigma_{\infty,\lambda} \cap [-r_j,r_j]$; denote this cover by $\mathcal I_n$. By \eqref{eq:smallspec} and Lemma~\ref{l:bandcount}, we have
\[
\sum_{I \in \mathcal I_n}
|I|^{\alpha}
\leq
\left(\frac{1}{\pi} T_n \sqrt{\Lambda_n(\|V_0\|_\infty + \varepsilon) + r_n} + 3 \right) 2^\alpha e^{-\alpha T_n^{1/2}}
\]
Sending $\delta \downarrow 0$ (and hence $n \to \infty$), we have $h^{\alpha}(\Sigma_{\infty,\lambda} \cap [-r_j,r_j]) = 0$, which proves the claim.
\end{proof}

With Claim~\ref{cl:hdim} in hand, we have $h^\alpha(\Sigma_{\infty,\lambda}) = 0$ immediately. Since this holds for all $\alpha > 0$ and all $\lambda > 0$, the theorem follows.
\end{proof}


\begin{thebibliography}{00}

\bibitem{avila2009CMP} A.\ Avila, On the spectrum and Lyapunov exponent of limit-periodic Schr\"odinger operators, \textit{Commun.\ Math.\ Phys.}\ \textbf{288} (2009), 907--918.

\bibitem{avila2015JAMS} A.\ Avila, On the Kotani-Last and Schr\"odinger conjectures, \textit{J.\ Amer.\ Math.\ Soc.}\ \textbf{28} (2015), 579--616.

\bibitem{avilaGlobal1} A.\ Avila, Global theory of one-frequency Schr\"odinger operators I: stratifed analyticity of the Lyapunov exponent and the boundary of nonuniform hyperbolicity, \textit{preprint}.

\bibitem{avilaGlobal2} A.\ Avila, Global theory of one-frequency Schr\"odinger operators II: acriticality and finiteness of phase transitions for typical potentials, \textit{preprint}.

\bibitem{ADZ14} A.\ Avila, D.\ Damanik, Z.\ Zhang, Singular density of states measure for subshift and quasi-periodic Schr\"odinger operators, \textit{Commun.\ Math.\ Phys.}\ \textbf{330} (2014), 469--498.

\bibitem{AS81} J.\ Avron, B.\ Simon, Almost periodic Schr\"odinger operators. I.~Limit periodic potentials, \textit{Commun.\ Math.\ Phys.}\ \textbf{82} (1981), 101--120.

\bibitem{AS83} J.\ Avron, B.\ Simon, Almost periodic Schr\"odinger operators. II.~The integrated density of states, \textit{Duke Math.\ J.}\ \textbf{50} (1983), 369--391.

\bibitem{BDGL} I.\ Binder, D.\ Damanik, M.\ Goldstein, M.\ Lukic, in preparation.

\bibitem{Chul81} V.\ Chulaevskii, Perturbations of a Schr\"odinger operator with periodic potential (Russian). \textit{Uspekhi Mat.\ Nauk} \textbf{36} (1981), 203--204.

\bibitem{DG} D.\ Damanik, M.\ Goldstein, On the existence and uniqueness of global solutions for the KdV equation with quasi-periodic initial data, to appear in \textit{J.\ Amer.\ Math.\ Soc.}\ (arXiv:1212.2674).

\bibitem{D} P.\ Deift, Some open problems in random matrix theory and the theory of integrable systems, \textit{Integrable Systems and Random Matrices}, 419--430, Contemp.\ Math.\ \textbf{458}, Amer.\ Math.\ Soc., Providence, RI, 2008.

\bibitem{E} I.\ Egorova, The Cauchy problem for the KdV equation with almost periodic initial data whose spectrum is nowhere dense, Spectral Operator Theory and Related Topics, 181--208, \textit{Adv.\ Soviet Math.}\ \textbf{19}, Amer.\ Math.\ Soc., Providence, RI, 1994.

\bibitem{Falconer} K.\ Falconer, \textit{Techniques in Fractal Geometry}, John Wiley \& Sons, Ltd., Chichester, 1997.

\bibitem{FL} J.\ Fillman, M.\ Lukic, Spectral homogeneity of limit-periodic Schr\"odinger operators, to appear in \textit{J.\ Spectr.\ Theory} (arXiv:1502.05454).

\bibitem{G76} A.\ Gordon, On the point spectrum of the one-dimensional Schr\"odinger operator, \textit{Usp.\ Math.\ Nauk} \textbf{31} (1976), 257--258.

\bibitem{Lukic2013MMNP} M.\ Lukic, Derivatives of $L^p$ eigenfunctions of Schr\"odinger operators, \textit{Math.\ Model.\ Nat.\ Phenom.}\ \textbf{8} (2013) 170--174.

\bibitem{MC} S.\ Molchanov, V.\ Chulaevsky, The structure of a spectrum of the lacunary-limit-periodic Schr\"odinger operator, \textit{Functional Anal.\ Appl.}\ \textbf{18} (1984), 343--344.

\bibitem{M} J.\ Moser, An example of a Schr\"odinger equation with almost periodic potential and nowhere dense spectrum, \textit{Comment.\ Math.\ Helv.}\ \textbf{56} (1981), 198--224.

\bibitem{PT84} L.\ Pastur, V.\ A.\ Tkachenko, On the spectral theory of the one-dimensional Schr\"odinger operator with limit-periodic potential (Russian), \textit{Dokl.\ Akad.\ Nauk SSSR} \textbf{279} (1984) 1050--1053.

\bibitem{PT88} L.\ Pastur, V.\ Tkachenko, Spectral theory of a class of one-dimensional Schr\"odinger operators with limit-periodic potentials \textit{Trudy Moskov.\ Mat.\ Obshch.}\ \textbf{51} (1988), 114--168.

\bibitem{simon76} B.\ Simon, On the genericity of nonvanishing instability intervals in Hill's equation, \textit{Annales de l'I.H.P., Section A}, \textbf{24} (1976), 91--93.

\bibitem{simon95Annals} B.\ Simon, Operators with singular continuous spectrum: I. General operators, \textit{Ann.\ Math.}\ \textbf{141} (1995) 131--145.

\bibitem{simon96PAMS} B.\ Simon, Bounded eigenfunctions and absolutely continuous spectra for one-dimensional Schr\"odinger operators, \textit{Proc.\ Amer.\ Math.\ Soc.} \textbf{124} (1996) 3361--3369.

\end{thebibliography}
\end{document}